\documentclass[12pt]{amsart}
\usepackage{graphicx}
\usepackage{tikz}        
\usepackage{amssymb}
\usepackage{epstopdf}
\usepackage[mathscr]{eucal}
\usepackage{amsmath,amssymb,amscd}
\usepackage{color}
\usepackage{epsfig}
\usepackage{bbold}
\usepackage{hyperref}
\newtheorem{theorem}{Theorem}

\newtheorem{definition}{Definition}
\newtheorem{corollary}[theorem]{Corollary}

\newtheorem{proposition}[theorem]{Proposition}
\theoremstyle{definition}
\newtheorem{remark}[theorem]{Remark}

\def \mb{\mathbb}

\def \bf{\mathbf}

\def \R{\mb R}                 
\def \C{\mb C}                 

\def \a{\alpha}         
\def \b{\beta}           
\def \D{\Delta}         
\def \vp{\varphi}       
\def \th{\theta}       

\newcommand {\diag} {\text{diag}}
\def \S{\mb S}        
\def\v{{\bf v}}
\def\u{{\bf u}}
\def\m{{\bf m}}
\def\0{{\bf 0}}

\newcommand {\q} {\mathbf{q}}
\newcommand {\p} {\mathbf{p}}

\def \and{\mbox{and}}
\usepackage[top=4cm, bottom=4cm, left=3.5cm, right=3.5cm]{geometry}
\title{ Regular polygonal equilibrium configurations on $\S^1$ and stability of the associated relative equilibria  }
\begin{document}
\maketitle
\markboth{}{}
\author{\begin{center}
					{Xiang Yu$^1$  and Shuqiang Zhu$^2,$}\\
			\ \ \ \ \ \ \ \ \ \ \ \ 	\ \ \ \ \ \ \ \ \ \ \ \ 	\ \ \ \ \ \ \ \ \ \ \ \ 	\ \ \ \ \ \ \ \ \ \ \ \ 		\ \ \ \ \ \ \ \ \ \ \ \ 	\emph{In memoriam  of Florin Diacu}

		\bigskip
		$^1$School of Economic and Mathematics, Southwestern
		University of Finance and Economics, Chengdu 611130, P.R. China \\
		$^2$School of Mathematical Sciences, University of Science and Technology of China, Hefei 230026,   P.R. China\\
		yuxiang@swufe.edu.cn, zhus@ustc.edu.cn\\
	\end{center}

 \begin{abstract}
 		For the curved $n$-body problem in $\S^3$, we show that a regular polygonal configuration for $n$ masses on a geodesic  is an equilibrium configuration if and only if  $n$ is odd and the masses are equal. The equilibrium configuration is associated with 
 		a one-parameter family  (depending on the angular  velocity) of  relative equilibria, which take place on $\S^1$ embedded in $\S^2$. We then study the stability of the associated  relative equilibria on two invariant manifolds, $T^*((\S^1)^n\setminus\D)$ and  $T^*((\S^2)^n\setminus\D)$.  We show that they are Lyapunov stable on $\S^1$,  they  are
 		Lyapunov stable  on $\S^2$  if the absolute value of  angular  velocity is larger  than a certain value, and that they are linearly unstable on $\S^2$ if the absolute value of  angular  velocity is smaller   than that certain value. 
\end{abstract}
\vspace{2mm}

\vspace{8mm}

\section{introduction}\label{sec:intro}

The curved $n$-body problem studies the motion of particles interacting under the cotangent potential  in three-dimensional sphere and three-dimensional hyperbolic sphere. It is a natural extension of the Newtonian $n$-body problem.  It roots in the research of Bolyai and Lobachevsky. There are many researches in this area over the past two decades on the Kepler problem, two-body problem, relative equilibria, stability of periodic orbits,  etc.   For history and  recent advances,  one can refer to  Arnold et al. \cite{AKN06},  Borisov et al. \cite{BMK04} and Diacu et al \cite{DPS12-1, Dia13-1}.

In classical mechanics, a particle is in \emph{mechanical equilibrium configuration} if the net force on that particle is zero.  An \emph{equilibrium configuration} is a configuration for which all particles are in mechanical equilibrium. Equilibrium configurations do not exist in the Newtonian  $n$-body problem. However, they do exist in the curved   $n$-body problem in $\S^3$. They are critical points of the potential.  They lead to the equilibrium solutions, as well as families of relative equilibria. 
  
The purpose of this paper is to study regular polygonal equilibrium configurations in $\S^3$  and the stability of the associated relative equilibria. We show that a regular polygonal configuration for $n$ masses on the equator (denoted by $\S^1$ ) of a 2-dimensional great sphere (denoted by $\S^2$ ) is an equilibrium configuration if and only if  $n$ is odd and  the masses are equal.  Each of the equilibrium configurations leads to a   one-parameter family  (depending on the angular  velocity) of  relative equilibria on $\S^1$. Both $T^*((\S^1)^n\setminus\D)$ and  $T^*((\S^2)^n\setminus\D)$ are invariant manifolds of the Hamiltonian system. We  show that the family of relative equilibria  are  Lyapunov stable on $\S^1$, they are 
Lyapunov stable  on $\S^2$  if the absolute value of  angular  velocity is larger  than a certain value, and that they are linearly unstable on $\S^2$ if the absolute value of  angular  velocity is smaller   than that certain value.


In the Newtonian $n$-body problem, relative equilibria are related to  the   planar central configurations. The symmetry of central configurations always associates with  the symmetry of the  masses. 
 For example, Perko-Walter \cite{PW85} shows that  the regular $n$-gon is a central configuration if and only if all masses are equal.  The  regular $n$-gon central configurations always lead to linearly unstable relative equilibria (cf. Moeckel \cite{Moe95}, Roberts \cite{Rob99-1}).

In the curved $n$-body problem in $\S^3$, the stability of regular polygonal relative equilibria was first studied by Mart\'{i}nez-Sim\'{o} \cite{MS13}. 
They consider  relative equilibria of three equal masses moving on upper half of  $\S^2$ embedded in $\R^3$.  The masses are moving on a circle $x^2 +y^2 =r^2, r\in (0,1)$ and form 
an equilateral triangle viewed from the ambient space $\R^3$, \cite{DP11}. The angular velocity is determined by $r$, see Remark \ref{rem:bif}.  They find that the linear stability depends on the angular velocity. 
The stability of three-body relative equilibria on the equator  of $\S^2$ was studied by  Diacu-S\'{a}nchez Cerritos-Zhu \cite{DSZ16}. They 
find that relative equilibria of three masses (not necessarily equal) on the equator  are Lyapunov stable on the equator, Lyapunov stable on $\S^2$ if the absolute value of the  angular velocity is larger than a certain value. On the other hand, Stoica \cite{Sto18} investigated  the general  $n$-body problem on surface of revolution. For the equal masses case, if the potential is attractive,  she finds that regular polygonal relative equilibria on geodesic circles are  unstable if the angular momentum is smaller than a certain value, and are  stable otherwise within a four-dimensional invariant manifold. She also finds that there is typically a pitchfork bifurcation. Our work can be viewed as an extension of her stability result in the four-dimensional invariant manifold to the full phase space  for the case of gravitational $n$-body  problem on $\S^2$.

 \section{The curved $n$-body problem in $\S^3$ and  main results  }\label{sec:main}

 In this section,  we review  the curved $n$-body problem in $\S^3$,  discuss the equilibrium configurations and
 state the main results of this paper. Vectors are all column vectors, but written as row vectors in the  text.

 \subsection{The curved $n$-body problem in $\S^3$ and equilibrium configurations}

The  curved  $n$-body  problem in the three-dimensional sphere  studies  the motion of $n$ particles interacting  under the so-called cotangent potential.  There are researches in which the problem was set  up   with other  models  of  the three-dimensional sphere. Following Diacu  \cite{Dia13-1},  we use the unit  sphere in $\R^4$. That is,  
 $\S^3=\{ (x,y,z,w)\in \R^4| x^2+y^2+z^2+w^2=1  \}$.   The metric on $\S^3$ is  induced from the standard metric of $\R^4$.  The  distance   between two   point $\q_i$ and $\q_j$,    $d_{ij}$,  is   computed by
$\cos  d_{ij}= \q_i \cdot \q_j$, where $\cdot $ is the scalar product in $\R^4$.

The curved $n$-body problem in $\S^3$ is a Hamiltonian system   in $(\R^4)^n$ with  holonomic constraints. The Hamiltonian is 
\[   H= \sum_{i=1}^n \frac{||\p_i||^2}{2m_i} -U(\q),    \] 
where $\q=(\q_1, ..., \q_n)$, $\q_i\in \R^4$, $\p_i=m_i \dot{\q}_i$,  $U(\q)$ is the potential defined by $U=\sum m_im_j \cot d_{ij}$, and the constraints are $\q_i \cdot \q_i=1, i=1, ..., n$.  The potential implies that 
the   singularity set of the configuration space is  $\D=\cup_{1\le i<j\le n}\{\q\in (\S^3)^n\ \! | \!  \q_i=\pm \q_j\}$. 
The  equations of motion  are (cf.\cite{Dia13-1, DSZ17})
\begin{equation}\label{equ:main}
\begin{cases}
\dot {\q}_i =\p_i/m_i \cr 
\dot\p_i=\sum_{j=1,j\ne i}^n\frac{m_im_j [\q_j-\cos d_{ij}\q_i]}{\sin^3 d_{ij}}-(\p_i\cdot \p_i)\q_i/m_i\cr
\q_i\cdot \q_i=1, \ \ \ \ i=1,..., n.
\end{cases}
\end{equation}

\begin{definition}
	A configuration $\q\in (\S^3)^n\setminus \D$ is called an  equilibrium configuration if it is a critical point of the potential, i.e., $	\nabla_{\q_i}U(\q)=\0,\  i=1,...,n,$
	where  $\nabla U (\q)$ means the gradient of $U$. 
That is,  $\q$ is  an equilibrium configuration if 
	\begin{equation} \label{equ:F}
 \nabla_{\q_i} U (\q)=\sum_{j=1, j\ne i}^n \frac{m_im_j [\q_j-\cos d_{ij}\q_i]}{\sin^3 d_{ij}}=\0,\  i=1,...,n. 
	\end{equation}
\end{definition}
Those configurations are first introduced  by Diacu \cite{Dia13-1} in the name \emph{fixed-points}. Then they are called \emph{special central configurations} by Diacu-Stoica-Zhu \cite{DSZ17}. Since the name ``central configuration''  does not  suit them very well, namely, they do not lead to total collision motion as in the Newtonian $n$-body problem, we call them 
\emph{ equilibrium configurations}.  
Note that the set of  equilibrium configurations have $O(4)$ symmetry. The system \eqref{equ:F} can be written in another  equivalent form. 
\begin{proposition}\label{prop:SCCE1} 
	An $n$-body  configuration $\q$ in  $\S^3$ is an equilibrium configuration  if   there are $n$ real  constants $\lambda_1, ..., \lambda_n$ such that
	\begin{equation} \label{equ:SCCE1}
	\sum _{j\ne i, j=1}^n \frac{m_j\q_j}{\sin^3 d_{ij}}- \lambda_i\q_i=0,  \ i=1, ..., n. 
	\end{equation}
\end{proposition}

\begin{proof}
	Suppose that system \eqref{equ:SCCE1} hold. Multiplying the $i$-th one  with $\q_i$, we obtain 
	\[   \lambda_i = \lambda_i \q_i \cdot \q_i = \sum _{j\ne i, j=1}^n \frac{m_j\q_j\cdot \q_i}{\sin^3 d_{ij}}=   \sum _{j\ne i, j=1}^n \frac{m_j\cos d_{ij}}{\sin^3 d_{ij}}.   \]
	Thus, system \eqref{equ:SCCE1} are equivalent to  system \eqref{equ:F}. 
	\end{proof}

The equilibrium configurations obviously lead to equilibrium solutions $\q(t)=\q(0), t\in \R$. They also lead to other simple motions.  Relative equilibrium is a phase curve  that is at the same time a one-parameter orbit of the action of the symmetry  group of the system.   
For the curved $n$-body problem in $\S^3$, with the unite sphere model, the symmetry group is $O(4)$. 
Each of one-parameter subgroups  of $O(4)$  is a  conjugate to the following subgroup
\[ A_{\alpha, \beta}(t)= \begin{bmatrix}
\cos\alpha t & -\sin \alpha t&0 &0\\
\sin \alpha t&\cos\alpha t&0&0\\
0&0&\cos\beta t & -\sin \beta t\\
0&0&\sin \beta t&\cos\beta t
\end{bmatrix}, \a, \b \in \R,   \]
which justifies the following definition. 

\begin{definition}
	For the curved $n$-body problem in $\S^3$,  a   solution in the form of $A_{\alpha, \beta}(t)\q(0)$ is called  a \emph{relative equilibrium}. 
\end{definition}



\begin{proposition}[\cite{DSZ17}]\label{cor:re&cc}
	For any  equilibrium configuration  $\q$,  there is a one-parameter family of relative equilibria associated to it, namely, $A_{\a, \a}(t)\q$ for any $\a\in \R$. Further more, if the equilibrium configuration  lies  on the union of the two great circles, $x^2+y^2=1,$ and $ z^2+w^2=1$, then there is a two-parameter family of relative equilibria associated to it, namely, $A_{\a, \b}(t)\q$ for any $\a, \b \in \R$.
\end{proposition}

There are actually more relative equilibria related to one equilibrium configurations $\q$.  Let $\tau\in O(4)$. Then obviously,  $\tau \q=(\tau \q_1, ..., \tau \q_n)$ is also an equilibrium configuration.  Thus, $A_{\a, \a}(t)\tau \q$ is a relative equilibrium for any $\a \in \R$.  Thus, there is a  $7$-parameter family of relative equilibria related to $\q$.  Suppose that $\q$ is an equilibrium configuration on the great circle  $x^2+y^2=1$. Then, $A_{\a, \b}(t)\q$ is a relative equilibrium  for any $\a, \b \in \R$, which  is equivalent to  $A_{\a, 0}(t)\q$ for  $\a \in \R$. Also,  $A_{\a, \a}(t)\tau (\q)$ is a relative equilibrium  for any $\a \in \R$ and any  $\tau \in O(4)$.

\subsection{Main results}

Denote by  $\S^1$ and $\S^2$ the specified great circle $\{ (x,y,z,w)\in \R^4|   x^2+y^2 =1, z=w=0  \}$ and the specified great two-sphere  $\{ (x,y,z,w)\in \R^4|   x^2+y^2+z^2 =1, w=0  \}$ respectively. Then $\S^1$ is the equator of $\S^2$ embedded in $\R^3=\{ (x,y,z)\}$.  In this paper, we focus on the regular polygonal equilibrium configurations on a great circle and the associated relative equilibria. 
 By symmetry, we assume that the equilibrium configurations are on $\S^1$.    The questions we are going to discuss are
\begin{enumerate}
	\item Let $\bar \q$ be a regular polygonal configuration (viewed from $\R^3$) on $\S^1$. To form an equilibrium configuration,  is it necessary that the masses are equal?
	\item The associated relative equilibria $ \q(t)=A_{\a, 0}(t)\bar \q$ take place  on $\S^1$, the equator of $\S^2$.  By the equations of motion \eqref{equ:main},   both  $T^*((\S^1)^n\setminus\D)$ and  $T^*((\S^2)^n\setminus\D)$ are invariant manifolds of the Hamiltonian system.  	Consider the relative equilibria on $T^*((\S^1)^n\setminus\D)$, are they  stable? 	Consider the relative equilibria on $T^*((\S^2)^n\setminus\D)$, are they  stable? 
\end{enumerate}

 We only study the case that the number of vertices  of the regular  polygon is odd. If the number is even, then there are pair of particles opposite to each other, i.e., $\q_i=-\q_j$. Then the configuration belongs to $\D$.  
 We  use the spherical coordinate system $(\vp, \th)$ for $\S^2$, $ \vp\in [0, 2\pi), \th \in [0, \pi)$. Recall that the Cartesian coordinates and the spherical coordinates are related by $(x,y,z)=(\sin \th \cos\vp, \sin\th \sin \vp, \cos \th)$.  Then,   $(\S^2)^n$ is  parametrized by $(\vp_1, ..., \vp_n, \th_1, ..., \th_n)$, and $\S^1$ is  parametrized by $( \vp, \frac{\pi}{2})$. 



 \begin{proposition}
 An $n$-body ($n$ is  not necessarily odd)  configuration on $\S^1$ is an  equilibrium configuration  if
 \begin{equation} \label{equ:SCCE2}
 0= \sum_{j\ne i, j=1}^n \frac{m_j\sin (\varphi_j-\varphi_i) }{\sin ^3 d_{ij}}, \ i=1, ..., n. 
 \end{equation}

 \end{proposition}
 \begin{proof}
 	We  identify $\S^1$ as the unit circle of the complex plane, i.e., $(x,y)= e^{\sqrt{-1} \vp}$.  Then,  a configuration on $\S^1$ is given by $\q=(e^{\sqrt{-1} \vp_1}, ...,  e^{\sqrt{-1} \vp_n})$, $  \ 0\le\varphi_1< \varphi_2<\cdots<\varphi_n\le 2\pi, \vp_1+2\pi \ne \vp_n. $
 The system  \eqref{equ:SCCE1} reads
 $	\sum _{j\ne i, j=1}^n \frac{m_je^{\sqrt{-1} \vp_j}}{\sin^3 d_{ij}}=\lambda_i e^{\sqrt{-1} \vp_i}, i=1, ..., n. $  They lead to  $	\sum _{j\ne i, j=1}^n \frac{m_je^{\sqrt{-1} (\vp_j-\vp_i)}}{\sin^3 d_{ij}}=\lambda_i \in \R, i=1, ..., n$.  That is,   the imaginary part of the left hand side of each equation  is zero. This completes the proof.  
 	\end{proof}


Now let  $n$ be  an odd number  greater  than $1$. Consider the regular $n$-gon configuration on  $\S^1$ given by  $\bar \q=( \frac{2\pi}{n}, ..., \frac{2n\pi}{n}, \frac{\pi}{2}, ..., \frac{\pi}{2})$. 

 \begin{theorem}\label{thm:masses}
 Let $n$ be  an odd number  greater  than $1$,  the regular $n$-gon configuration on $S^1$ is an equilibrium configuration  if and only if $m_1=m_2=...=m_n$
 \end{theorem}


\begin{proposition}\label{prop:main}
Let $n=2p+1, p\ge 1$. Let  $\bar \q$ be the   regular $n$-gon equilibrium configuration on $\S^1$ with masses $m_1=...=m_n=1$. In the coordinate system $(\vp_1, ..., \vp_n, \th_1, ..., \th_n)$,  the  Hessian matrix of $U$  at $\bar \q$ is  block diagonal, i.e., 
\[   D^2 U(\bar \q) = \diag\{   \frac{\partial^2 U}{\partial \vp_i\partial\vp_j}(\bar \q),  \frac{\partial^2 U}{\partial \th_i\partial\th_j} (\bar \q)\}.  \] 
The eigenvalues of the first block  $\left[\frac{\partial^2 U}{\partial \vp_i\partial\vp_j} \right]$ consist of one zero, and $n-1$ negative numbers.  The eigenvalues of the second block  $\left[\frac{\partial^2 U}{\partial \th_i\partial\th_j} \right]$ consist of two zeros, and $n-2$ positive numbers. The maximal eigenvalue of the second block  is $ 2\sum_{j=1}^p \frac{1-\cos j\frac{2\pi}{n}}{\sin^3 j\frac{2\pi}{n}}$.
\end{proposition}

The stability of relative equilibria is often defined by the stability of the corresponding 
   equilibria  of the flow on the reduced phase space. 
    The usual practice is to compute the eigenvalues of the unreduced system and then to skip the non-relevant eigenvalues at the end, \cite{MS13,Moe95, Rob99-1}. However, we would like to do the computation in the reduced system.  
Denote by 
\begin{align*}
&\vec{\vp}=(\vp_1, ..., \vp_n),   \     & \vec{p_\vp}=(p_{\vp_1}, ..., p_{\vp_n}),   \\ &\vec{\th}=(\th_1, ..., \th_n),   \  &\vec{p_\th}=(p_{\th_1}, ..., p_{\th_n}). 
\end{align*}

We first do the reduction   for the $n$-body problem on $\S^1$.  In this case,  the potential $U$ depends on $\vec{\vp}$ only and  the Hamiltonian system can be written as 
\[   H(\vec{\vp}, \vec{p_{\vp}})=\sum_{i=1}^n \frac{1}{2m_i} p^2_{\vp_i} -U(\vec{\vp}), \   T^*(\S^1)^n, \ \ \omega =  d(\sum_{i=1}^n\vp_i  dp_{\vp_i}).   \]
Obviously, the group $SO(2)$ has a Hamiltonian action   on the phase space by $(\vec{\vp},\vec{p_{\vp}})\mapsto(\vec{\vp} + \vec{s},\vec{p_{\vp}})$, where $\vec{\vp} + \vec{s}= (\vp_1+s, ..., \vp_n+s)$. 
The corresponding first integral is $J_1(\vec{\vp}, \vec{p_{\vp}})=\sum_{i=1}^n p_{\vp_i}$.  The action of  $SO(2)$ and the integral is analogous to the action of $\R^1$ on $\R^1$ and the corresponding integral, so we use the Jacobi coordinates \cite{MHO09}  of the Newtonian $n$-body problem  to do the reduction. 

Let $\mu_k=\sum_{i=1}^k m_i$, and $\mathtt{M}_k= \frac{m_k \mu_{k-1}}{\mu_k}$, ($\frac{1}{\mathtt{M}_k} = \frac{1}{m_k} +\frac{1}{\mu_{k-1}}$). Denote by 
\[ \vec{u}=(u_2, u_3 ..., u_n, g_n),   \  \   \vec{v}=(v_2, v_3, ..., v_n, G_n).\]
Consider the canonical transformation  from $(\vec{\vp}, \vec{p_\vp})$ to $(\vec{u}, \vec{v})$ given by the generating function $F_2(\vec{v}, \vec{\vp}) = \mathcal A\vec{\vp} \cdot \vec{v}=\mathcal A^T\vec{v} \cdot \vec{\vp},$
where 
\[  \mathcal A= \begin{bmatrix} -1&1&0&0&\dots&0\\
-\frac{m_1}{\mu_2}&-\frac{m_2}{\mu_2}&1&0&\dots&0\\
-\frac{m_1}{\mu_3}&-\frac{m_2}{\mu_3}&-\frac{m_3}{\mu_3}&1&\dots&0\\
\vdots&\vdots&\vdots&\vdots&\ddots&0\\
\frac{m_1}{\mu_n}&\frac{m_2}{\mu_n}&\frac{m_3}{\mu_n}&\frac{m_4}{\mu_n}&\dots&\frac{m_n}{\mu_n}
\end{bmatrix}. \]
The explicit transformation  is 
\begin{equation} \label{equ:jac}
\vec{u}= \frac{\partial F_2}{\partial \vec{v}} = \mathcal A\vec{\vp}, \  \  \vec{p_\vp}= \frac{\partial F_2}{\partial \vec{\vp}} = \mathcal A^T\vec{v}. \  \  
\end{equation}
It is well-known that the Jacobi coordinates system has the following  properties (cf. \cite{MHO09}):  
\begin{equation} \label{equ:kin}
(\vec{p_\vp} )  ^T M \vec{p_\vp}  =  
\vec{v}^T  \mathcal AM \mathcal A^T \vec{v}  =  \vec{v}^T  \tilde M   \vec{v},  \ \ \mu_n g_n = \sum_{i=1}^n m_i \vp_i, \   G_n = \sum_{i=1}^n  p_{\vp_i}, 
\end{equation}
where $M =\diag\{ \frac{1}{m_1}, ..., \frac{1}{m_n} \}, \ \  \tilde M =  \diag\{ \frac{1}{\mathtt{M}_2}, ..., \frac{1}{\mathtt{M}_n}, \frac{1}{\mu_n} \}. $
Note that the potential $U$ does not depend on $g_n$. Suppose $\vec{\vp}$ corresponds to $(u_2, ..., u_n, g_n)$. Then using the transform \eqref{equ:kin}, we see that $\vec{\vp}+\vec{s}$ corresponds to $(u_2, ..., u_n, g_n+s)$. So
\[ \frac{\partial U}{\partial g_n}=\lim_{s\to 0} \frac{U(u_2, ..., u_n, g_n+s)-U(u_2, ..., u_n, g_n) }{s}=\lim_{s\to 0} \frac{U(\vec{\vp} +\vec{s}) -U(\vec{\vp})}{s}=0. 
  \]
 Hence, with the Jacobi coordinates,  the Hamiltonian function can be written as
\begin{equation*} 
\begin{split}
H(\vec{u}, \vec{v}) 
=  \sum_{i=2}^n \frac{v_i^2}{2\mathtt M_i}   + \frac{G_n^2}{2\mu_n}  -U(u_2, ..., u_n).
\end{split}
\end{equation*}

Consider the reduced space $J_1^{-1}(c)/SO(2)$. Obviously, $(u_2, ..., u_n, v_2, ..., v_n)$ can serve as a canonical coordinates system of the symplectic sub-manifold. The reduced Hamiltonian system  is 
\[   H_1=\sum_{i=2}^n \frac{v_i^2}{2\mathtt M_i}    -U(u_2, ..., u_n), \   J_1^{-1}(c)/SO(2), \ \ \omega_1 =  d(\sum _{i=2}^n u_i d v_i),  \]
where we have neglected the constant term. Consider  a relative equilibrium   on $\S^1$ in the $(\vec{\vp}, \vec{p_\vp})$ coordinates system, 
$ \vp_1(t) = \vp_1 + \alpha t,..., \vp_n(t) = \vp_n + \alpha t,  p_{\vp_1}(t)=m_1\alpha, ..., p_{\vp_n}(t)=m_n \alpha.$
It is easy to check that  the motion corresponds to the equilibrium 
\[ \vp_2-\vp_1, ..., \vp_{n} - \frac{\sum_{i=1}^{n-1} m_i \vp_i}{\mu_{n-1}}, 0 ,..., 0,   \]
of the reduced Hamiltonian system.  For the relative equilibrium  $A_{\a, 0}(t)\bar \q$ we are interested, the momentum is $J_1=n\a$, and  the corresponding  equilibrium is 
\[  X_\a = ( \frac{2}{n}\pi,  ...,  \frac{k+1}{n}\pi, ..., \pi,  0, ..., 0 ), \a\in \R.    \]
 \begin{theorem}\label{thm:stas1}
	Let $n=2p+1, p\ge 1$.  Let  $\bar \q$ be the  regular $n$-gon equilibrium configuration on $\S^1$.  Then the   equilibrium $X_a$  of the reduced Hamiltonian system $(H_1, J_1^{-1}(n\a)/SO(2), \omega_1)$ 
	is Lyapunov  stable for all $\a \in \R$.
\end{theorem}

We now do the reduction for the $n$-body problem on $\S^2$.  The complete reduction for the  two-dimensional case is not easy, \cite{BMK04}. For our purpose, we will do a partial reduction.  In this case, 
 the  Hamiltonian system is 
\[   H(\vec{\vp}, \vec{\th},  \vec{p_{\vp}}, \vec{p_\th})=\sum_{i=1}^n (\frac{p^2_{\vp_i}}{2m_i\sin^2 \th_i}   +  \frac{p^2_{\th_i}}{2m_i} ) -U(\vec{\vp}, \vec{\th}), \   T^*(\S^2)^n, \ \ \omega =  d(\sum_{i=1}^n\vp_i  dp_{\vp_i}+ \th_i d p_{\th_i}).   \]
Again,   the group $SO(2)$ acts on the phase space by  
$  (\vec{\vp}, \vec{\th},  \vec{p_{\vp}}, \vec{p_\th})\mapsto (\vec{\vp}+\vec{s}, \vec{\th},  \vec{p_{\vp}}, \vec{p_\th}), $
and the corresponding first integral is $J_2(\vec{\vp}, \vec{\th},  \vec{p_{\vp}}, \vec{p_\th})= \sum_{i=1}^n p_{\vp_i}$. We again consider the 
canonical transformation given by \eqref{equ:jac}, i.e., $\vec{u} = \mathcal A \vec{\vp}, \  \vec{p_{\vp}} = \mathcal A^T \vec{v}.$
The relations $\mu_n g_n = \sum_{i=1}^n m_i \vp_i, \   G_n = \sum_{i=1}^n  p_{\vp_i} $  still hold \cite{MHO09}. However, the kinetic energy is more  complicated than the previous case. Let  $S=\diag\{ \frac{1}{\sin \th_1}, ...,  \frac{1}{\sin \th_n} \}. $ Then, the first part of  the kinetic energy 
is $\frac{1}{2}\vec{p_{\vp}} ^T  SMS \vec{p_{\vp}} =\frac{1}{2}\vec{v}^T  \mathcal A  SMS \mathcal A^T  \vec{v}$. Denote by $P$ the matrix $\mathcal A  SMS \mathcal A^T$.  Then, $P$ is not diagonal, and by direct computation we find that  the elements are 
 \begin{equation}  \label{equ:kinetic}\begin{cases} P_{kk}=\frac{1}{\mu^2_k} [  \frac{m_1}{\sin^2 \th_1} +...+ \frac{m_k}{ \sin^2 \th_k} ] +\frac{1}{m_{k+1}\sin^2 \th_{k+1}}, &k \ne n,  \\
 P_{nn}=\frac{1}{\mu^2_n} [  \frac{m_1}{\sin^2 \th_1} +...+ \frac{m_n}{ \sin^2 \th_n} ],   & \   \\
 P_{kl}= \frac{1}{\mu_k\mu_l} [  \frac{m_1}{\sin^2 \th_1} +...+ \frac{m_k}{ \sin^2 \th_k} ] - \frac{1}{\mu_l \sin^2  \th_{k+1}},  & k<l<n,    \\
 P_{kn}=-\frac{1}{\mu_k\mu_n} [  \frac{m_1}{\sin^2 \th_1} +...+ \frac{m_k}{ \sin^2 \th_k} ]  +\frac{1}{\mu_n\sin^2 \th_{k+1}},   & k<n.    \end{cases} 
\end{equation}    
As in the previous case, the potential $U$ does not depend on $g_n$. The reduced space $J_2^{-1}(c)/SO(2)$ can be parametrized by $(u_2, ..., u_n, \vec{\th}, v_2, ..., v_n, \vec{p_\th})$.  The reduced Hamiltonian system  is 
\[   H_2=
 \frac{1}{2}\vec{v}^T P  \vec{v} +  \sum_{i=1}^n \frac{p^2_{\th_i}}{2m_i}   -U(u_2, ..., u_n, \vec{\th}), \   J_2^{-1}(c)/SO(2), \ \ \omega_2 =  d(\sum _{i=2}^n u_i d v_i+ \th_i dp_{\th_i}).   \]
The reduced system might not be useful for a general problem, but it works in our problem.  

For the relative equilibrium  $A_{\a, 0}(t)\bar \q$ we are studying, the momentum is $J_2=n\a=G_n$, and  the corresponding  equilibrium is 
\[  Y_\a = ( \frac{2}{n}\pi, \frac{3}{n}\pi, ...,  \frac{n}{n}\pi, \frac{\pi}{2},...,  \frac{\pi}{2},  0, ..., 0, 0, ..., 0 ), \a\in \R.    \]

\begin{theorem}\label{thm:stas2}
	Let $n=2p+1, p\ge 1$.  Let  $\bar \q$ be the  regular $n$-gon equilibrium configuration on $\S^1$ with masses $m_1=...=m_n=1$.  Then the   equilibrium $Y_a$  of the reduced Hamiltonian system $(H_2, J_2^{-1}(n\a)/SO(2), \omega_2)$ 
	is linearly   unstable if  $\a ^2<  2\sum_{j=1}^p \frac{1-\cos j\frac{2\pi}{n}}{\sin^3 j\frac{2\pi}{n}}$, and  is Lyapunov stable if  $\a ^2>  2\sum_{j=1}^p \frac{1-\cos j\frac{2\pi}{n}}{\sin^3 j\frac{2\pi}{n}}$. 
\end{theorem}

\begin{remark}
	Though the equilibria $X_\a (Y_\a)$ are  stable in the reduced system, the corresponding relative equilibrium is obviously not stable in the unreduced system, since any relative equilibrium can be perturbed in such a way that the configuration only rotates more quickly.  This is typical for relative equilibria stable in the reduced system, see Patrick \cite{Pat92}. 
\end{remark}
\section{proof of the main results}\label{sec:proof}

To prove  Theorem \ref{thm:masses} and Proposition \ref{prop:main}, we need  the following  property of  \emph{circulant matrices}.
An $n \times  n$ matrix $C = (c_{kj})$  is called circulant if $c_{kj} = c_{k-1, j-1},$  where $c_{0, j}$  and $c_{k, 0}$   are identified with $c_{n, j}$  and $c_{k, n}$, respectively.  For any $n\times n$ circulant matrix $C=(c_{ij})$, the eigenvectors and the corresponding eigenvalues  are 
\begin{equation}\label{equ:circu}
\v _k= (1, \rho_{k-1},\rho_{k-1}^2,\ldots, \rho_{k-1}^{n-1}),  \ \lambda_k = \sum_{j=1}^n c_{1j} \rho_{k-1}^{j-1},  \  k=1, ..., n,  
\end{equation} 
where $\rho_{k}$ is the $k$-th root of unity $\rho_k=e^{\sqrt{-1}  \frac{2k\pi}{n}}$.  
We also need  the identity 
\begin{equation}\label{equ:trig}
   \sum_{j=1}^{m}\sin 2ja= \frac{  \sin [a (m+1)] \sin a m }{\sin a }. 
\end{equation}

Let us introduce some \textbf{notations}. Assume that $n=2p+1, p\ge 1$.  Denote by  $\phi$   the angle $\frac{2\pi}{n}$,    
and by $\v _k$ the $k$-th eigenvector of circulant matrix defined above  for $ k=1, ..., n$. Recall that the regular $n$-gon configuration is given by  $\bar \q=( \frac{2\pi}{n}, ..., \frac{2n\pi}{n}, \frac{\pi}{2}, ..., \frac{\pi}{2} )$ and that  $\rho_k=\cos k \phi + \sqrt{-1} \sin k \phi$.  

\subsection{Proof of Theorem \ref{thm:masses}}
\begin{proof}
	If the masses are equal,  system \eqref{equ:SCCE1}  is  obviously satisfied. Then,  the configuration $\bar \q$ is an equilibrium configuration.   Now we show that  it also necessary. 
	For the regular $n$-gon, $d_{ij}=\min\{ |\vp_i-\vp_j|, 2\pi-|\vp_i-\vp_j|\}$, so system \eqref{equ:SCCE2} is   $\sum_{j\ne k} m_j \frac{\sin (j-k)\phi }{|\sin (j-k) \phi  |^3} =0, \ k=1, ..., n$. 
	It can be written as 
	$ B\m =\0$, where $\m =(m_1, ..., m_n)$ and  $B$ has elements
\[  b_{kj} =  \begin{cases}   \frac{\sin (j-k)\phi }{|\sin (j-k) \phi  |^3}     \  \  & {\rm for }
\ j\ne k,\\
0  \    \ \   \    &{\rm for }
\ j= k.   \end{cases} \]	
Note that $B$ is circulant and skew-symmetric, so its eigenvalues can be computed by formula \eqref{equ:circu} and are purely imaginary. 
Denote  the eigenvalues of $B$ by $\sqrt{-1}\Gamma_1 ..., \sqrt{-1}\Gamma_n.$ Then, $\Gamma_k  = \frac{1}{\sqrt{-1}} \sum_{j=1}^n b_{1j} \sin (k-1)(j-1)\phi $ 
, and 
 \begin{equation} \notag
 \begin{split}
 \Gamma_k  
 &= \sum_{j=1}^{p}\frac{\sin  j \phi }{|\sin j  \phi  |^3} \sin j(k-1) \phi + \sum_{j=p+1}^{2p}\frac{\sin  j \phi }{|\sin j  \phi  |^3} \sin j(k-1) \phi \\
 &= 2  \sum_{j=1}^{p}\frac{\sin j(k-1) \phi }{\sin^2 j  \phi  } , \ k=1, ..., n.
 \end{split}
 \end{equation}

We first show that  $ \Gamma_k\ne 0$ for $k=2, ..., n$. 
Note that $\rho_{k-1}$ is   the complex conjugation of $\rho_{n-(k-1)} =\rho_{n+2-k-1}$ and that  $B$ is  real.  We obtain 
$ -\Gamma_{n+2-k}= {\Gamma}_k. $
Thus,  it is enough to show that none of the the following numbers is zero,
\[  \Gamma_2, \  \Gamma_4, \  \Gamma_6, \ ..., \  \Gamma_{2p}.     \]
We claim that  the sequence  $\{ \Gamma_2, \  \Gamma_4, \  \Gamma_6, \ ..., \  \Gamma_{2p}  \}$  is  concave. By  elementary trigonometric identities and the formula \eqref{equ:trig}, we obtain 
 \begin{align*}
 &\Gamma_{k+2}-\Gamma_k=4\sum_{j=1}^{p}\frac{\cos k  j \phi}{\sin j\phi}, \\
 &\Gamma_{k+4}-2\Gamma_{k+2}+ \Gamma_k= -8\sum_{j=1}^{p}\sin (k+1)  j \phi= -8 \frac{  \sin \frac{(k+1)(p+1)}{2}  \phi  \sin \frac{(k+1)p}{2} \phi }{\sin \frac{k+1}{2} \phi}\\
 &\ \ \ =4 \frac{ \cos \frac{(k+1)(2p+1)  }{2} \phi-\cos \frac{(k+1) }{2}\phi}{\sin \frac{k+1}{2} \phi}\\
 & \ \ \ =4 \frac{ (-1)^{k-1} -\cos \frac{(k+1) }{2}\phi}{\sin \frac{k+1}{2} \phi}.
\end{align*}
If $k$ is even and   $k\in [2, 2p-2]$,  then $\frac{k+1}{2} \phi\in(0, \pi)$ and $\Gamma_{k+4}-2\Gamma_{k+2}+ \Gamma_k<0$. Thus, 
the sequence $\{ \Gamma_2, \  \Gamma_4, \  \Gamma_6, \ ..., \  \Gamma_{2p}  \}$  is concave.

We check that  the two ends of the sequence are positive.  The first one 
$ \Gamma_2=\sum_{j=1}^{p}\frac{1}{\sin j  \phi}$  is obviously positive since $0<j\phi<\pi$ for $1\le j\le p$.   The second one
\[ \Gamma_{2p}= - \sum_{j=1}^{p}\frac{\sin 2j  \phi}{\sin^2j  \phi}  = - 2\sum_{j=1}^{p}\cot  j \phi  =- \sum_{j=1}^{p} [\cot  j \phi + \cot (p+1-j)\phi].   \]
Note that  $	(p+1-j)\phi=\pi -( j \phi -\frac{\phi}{2})$, we obtain 
\begin{equation} \label{equ:1}
\Gamma_{2p}= - 2\sum_{j=1}^{p}\cot  j \phi  =- \sum_{j=1}^{p} [\cot  j \phi - \cot (j-\frac{1}{2})\phi]>0.
\end{equation}
Since that the sequence $\Gamma_2, \  \Gamma_4, \  \Gamma_6, \ ..., \  \Gamma_{2p}$  is concave, and that the two ends are positive, we conclude  that none of  the numbers $\Gamma_2, \  \Gamma_4, \  \Gamma_6, \ ..., \  \Gamma_{2p}$ is zero. Thus, $\Gamma_k\ne 0$ for $k=2, ..., n$. 

Now let us return to show  that it is necessary that  $\m=m_1(1,..., 1)$ to have $B\m=\0$.  Note that  the $n$ eigenvectors of $B$,  $\v _1, ..., \v _n$ form a basis of $\C^n$.  There are $n$ complex constants $\delta_1, ..., \delta_n$ such that $\m =\sum_{i=1}^n \delta_k \v _k$.  So, we have 
\[ \0= B\m =  B  \sum_{i=1}^n  \delta_k \v_k =  \sum_{i=1}^n  \sqrt{-1}\delta_k \Gamma_k \v_k    \Rightarrow \Gamma_k \delta_k =0, \ k=1, ..., n. \]   
 Since $\Gamma_k\ne 0$ for $k=2, ..., n$,  we obtain  $ \delta_k= 0$ for $k=2, ..., n$, i.e.,  $\m=\delta_1 \v_1$. That is, $m_1=m_2=...= m_n$,  a remark that completes the proof.
	\end{proof}	
\begin{remark} \label{rem:Nearequilibrium configurations}
  Consider a configuration of odd bodies on $\S^1$ that is close to a regular polygon.  We can  find masses to form an equilibrium configuration  by solving a linear system $\tilde B \m=0$ 
 equivalent to  system \eqref{equ:SCCE2}.  Since the matrix $\tilde B$ is anti-symmetric and is close the  matrix $B$ in the above proof, we conclude that $\tilde B \m=0$ has a one-dimensional solution space and that the masses can be all positive.  Hence, there is an $(n-1)$-dimensional manifold of equilibrium configurations  in $(\S^1)^n/SO(2)$. 
  This is different from co-circular central configurations of the Newtonian $n$-body problem. For example, it is proved by Cors-Roberts that  the four-body co-circular central configurations  form a two-dimensional manifold \cite{ CR12}.	
\end{remark}

\subsection{Proof of Proposition  \ref{prop:main}} 
 \begin{proof}
 	We first show that the Hessian matrix of $U$ is block diagonal at the equilibrium configuration $\bar \q$, i.e.,  $\frac{\partial ^2 U}{\partial \th_i \partial \vp_j} |_{\bar{\q}}=0$ for all pairs of $(i, j)$.  Denote by $\q +k_i$ the coordinate $( \vp_1,  ..., \vp_n, \th_1, ..., \th_i+k, ...,  \th_n )$, by  $\q +h_j$ the coordinate $(\vp_1,  ..., \vp_j+h, ..., \vp_n, \th_1, ...,  \th_n)$. Then the mutual distances between the particles are the same for the two configurations, $\q +k_i$ and $\q -k_i$ if $\q$ is a configuration on $\S^1$. That is, $U(\q +k_i)= U( \q -k_i)$.  Hence, we obtain 
 	\begin{align*}
 	&\frac{\partial ^2 U}{ \partial \th_i \partial \vp_j  } |_{\bar{\q}}=\lim _{k \to 0} \frac{1}{2k}(  \frac{\partial  U}{\partial \vp_j } |_{\bar{\q}+k_i}    -  \frac{\partial  U}{\partial \vp_j  } |_{\bar{\q} -k_i}    )\\
 	&=\lim _{(h,k)\to (0,0)} \frac{1}{2hk}  (   U(\bar{\q}+k_i+ h_j)  - U(\bar{\q}+k_i) 
 	- U(\bar{\q}-k_i + h_j)  + U(\bar{\q}-k_i)  )\\
 	&=0. 
  	\end{align*}
 We compute the elements of the two blocks of the Hessian matrix of $U$.   	Recall that the masses are $m_1=...=m_n=1$, $U=\sum m_im_j\cot d_{ij}$ and that $\cos d_{ij}= \cos \th_i \cos \th_j+ \sin \th_i \sin \th_j\cos(\vp_i-\vp_j)$. By direct computation, we obtain \cite{DSZ16}, 
\begin{equation}\label{equ:matrix}
	\begin{cases}
&	\frac{\partial U}{\partial \vp_i} = \sum_{j\ne i} \frac{-\sin\th_i \sin \th_j \sin (\vp_i-\vp_j)}{\sin^3 d_{ij}}, \   \frac{\partial U}{\partial \th_i} = \sum_{j\ne i} \frac{-\sin\th_i \cos \th_j + \cos \th_i \sin\th_j \cos (\vp_i-\vp_j)}{\sin^3 d_{ij}},  \\
&\frac{\partial ^2U}{\partial \vp_i \partial \vp_j}= \frac{-3\cos d_{ij} \sin^2\th_i \sin^2 \th_j \sin^2 (\vp_i-\vp_j) + \sin^2 d_{ij}  \sin\th_i \sin \th_j \cos (\vp_i-\vp_j) }{\sin^5 d_{ij}} ,\\
&\frac{\partial ^2U}{\partial \vp_i ^2}= \sum_{j\ne i}\frac{3\cos d_{ij} \sin^2\th_i \sin^2 \th_j \sin^2 (\vp_i-\vp_j) -\sin^2 d_{ij}  \sin\th_i \sin \th_j \cos (\vp_i-\vp_j) }{\sin^5 d_{ij}} ,\\
&\frac{\partial ^2U}{\partial \th_i \partial \th_j}= \frac{3\cos d_{ij} ( -\cos\th_i \sin \th_j+\sin\th_i\cos\th_j \cos (\vp_i-\vp_j) )( -\sin\th_i \cos \th_j+\cos\th_i\sin\th_j \cos (\vp_i-\vp_j) )}{\sin^5 d_{ij}} \\
&\ \   \ \ \ \ \ \ +\frac{ \sin^2 d_{ij} ( \sin\th_i \sin \th_j+\cos\th_i\cos\th_j \cos (\vp_i-\vp_j) )}{\sin^5 d_{ij}} ,\\
&\frac{\partial ^2U}{\partial \th_i ^2}= \sum_{j\ne i}\frac{3\cos d_{ij} ( -\sin\th_i \cos \th_j+\cos\th_i\sin\th_j \cos (\vp_i-\vp_j) )^2-\sin^2 d_{ij} \cos d_{ij}  }{\sin^5 d_{ij}}. 
\end{cases}
\end{equation}	
 Part 1.   The eigenvalues of $\left[\frac{\partial^2 U}{\partial \varphi_i\partial\varphi_j} \right]$ at $\bar \q$.   Note that $\th_i=\frac{\pi}{2}, i=1, ..., n$ and $d_{ij}=\min\{ |\vp_i-\vp_j|, 2\pi-|\vp_i-\vp_j|\}$. 
 By equation \eqref{equ:matrix}, the block  has elements 
\begin{equation} \label{equ:matrixb}
\frac{\partial^2 U}{\partial \varphi_k\partial\varphi_j} =  \begin{cases}   \frac{-2\cos  (j-k)\phi  }{|\sin (j-k)\phi |^3}    \  \  & {\rm for }
\ j\ne k,\\
\sum_{i\ne k}  \frac{2\cos  (i-k)\phi  }{|\sin (i-k)\phi |^3}
\    \ \   \    &{\rm for }
\ j= k,  \end{cases}
\end{equation}
Note that the block is circulant and symmetric. So its eigenvalues, denoted by $\Phi_1, ..., \Phi_n$,   can be computed by formula \eqref{equ:circu} and are real.  Then, 
\begin{equation*}
\begin{split}
\Phi_k&=  \sum_{j= 1}^{2p}\frac{2 \cos j \phi}{|\sin j \phi|^3} -  \sum_{j= 1}^{2p}\frac{2 \cos j \phi}{|\sin j \phi|^3} \cos (k-1)j\phi\\
&= 4  \sum_{j=1}^{p}\frac{\cos j \phi}{\sin^3j  \phi  }(1-\cos (k-1) j\phi) , \ k=1, ..., n. 
\end{split}
\end{equation*}
The first eigenvalue $\Phi_1$ is $0$,  which reflects the $SO(2)$ symmetry of the equilibrium configuration  on $\S^1$.  
Note that $\rho_{k-1}$ is   the complex conjugation of $\rho_{n-(k-1)} =\rho_{n+2-k-1}$ and that $\left[\frac{\partial^2 U}{\partial \varphi_i\partial\varphi_j} \right]$  is  a real matrix. We have 
 $ \Phi_{n+2-k} = \Phi_k. $
 Thus,  it is enough to study just the following eigenvalues
 \[  \Phi_2, \  \Phi_4, \  \Phi_6, \ ..., \  \Phi_{2p}.     \]

Firstly,  we claim that the sequence
$\{\Phi_4-\Phi_2,   \ \Phi_6-\Phi_4, \ ...,\ \Phi_{2p+2}-\Phi_{2p} \}$
 is  concave. By  formula \eqref{equ:trig} and other elementary trigonometric identities, we have
 \begin{align*}
& \Phi_{k+2}-\Phi_k=8\sum_{j=1}^{p}\frac{\sin k  j \phi\cos j \phi}{\sin^2 j\phi}, \  \Phi_{k+4}-2\Phi_{k+2}+ \Phi_k
 =16\sum_{j=1}^{p}\frac{\cos (k+1)  j \phi\cos j \phi}{\sin j\phi}, \\
 &\Phi_{k+6}-3\Phi_{k+4}+ 3\Phi_{k+2} -\Phi_k= -32\sum_{j=1}^{p} \sin (k+2)  j \phi\cos j\phi\\
 & \ \ \ = -16\sum_{j=1}^{p}\left(  \sin (k+3)j\phi   +  \sin (k+1)j\phi   \right)\\
 &  \ \ \ = -16\left(  \frac{\sin \frac{(k+3)(p+1)}{2}\phi \sin \frac{(k+3)p}{2}\phi   }{\sin \frac{k+3}{2}\phi } +\frac{\sin \frac{(k+1)(p+1)}{2}\phi \sin \frac{(k+1)p}{2}\phi  }{\sin \frac{k+1}{2}\phi }       \right)\\
 & \ \ \ = 8\left(  \frac{\cos \frac{(k+3)n}{2}\phi  -\cos \frac{k+3}{2}\phi  }{\sin \frac{k+3}{2}\phi } +\frac{\cos \frac{(k+1)n}{2}\phi -\cos \frac{k+1}{2}\phi  }{\sin \frac{k+1}{2}\phi }       \right)\\
 & \ \ \ =8\left[  \frac{ (-1)^{k+1}-\cos \frac{k+3}{2}\phi }{\sin \frac{(k+3)\pi}{n} }  +   \frac{ (-1)^{k+1}-\cos \frac{k+1}{2}\phi }{\sin \frac{(k+1)\pi}{n} }   \right]
 \end{align*}
 Thus,  	$\Phi_{k+6}-3\Phi_{k+4}+ 3\Phi_{k+2} -\Phi_k<0$  if $k$ is even and    $k\in [2, 2p-4]$, i.e., the sequence $\{\Phi_4-\Phi_2,   \ \Phi_6-\Phi_4, \ ...,\ \Phi_{2p+2}-\Phi_{2p} \}$  is concave.

 Secondly, note that  the two ends of the sequence 
 are positive. Note that  $\Phi_{2p+2}-\Phi_{2p}=\Phi_1-\Phi_3 =-\Phi_3$. Since 
 $ \frac{\Phi_3}{4}=  \sum_{j=1}^{p}\frac{\cos j \phi}{\sin^3 j\phi } (1-\cos 2j\phi) = 2\sum_{j=1}^{p}\cot j\phi, $
which is negative according to  \eqref{equ:1}, we see that  the second end $\Phi_{2p+2}-\Phi_{2p}$ is positive.   
The first end  $\frac{\Phi_4-\Phi_2}{4}$ can be written as 
\[  
\sum_{j=1}^{p}\frac{\cos j \phi }{\sin^3 j\phi } (\cos j\phi-\cos 3j\phi)=  \sum_{j=1}^{p}\frac{\cos j \phi }{\sin^3 j\phi } (4\cos j\phi- 4\cos ^3j\phi)= 4\sum_{j=1}^{p}\frac{\cos^2 j \phi}{\sin j\phi }.   \]
So  the first end $\Phi_4-\Phi_2 $ is also positive.

Hence,  the sequence $\{\Phi_4-\Phi_2,   \ \Phi_6-\Phi_4, \ ...,\ \Phi_{2p+2}-\Phi_{2p} \}$  is positive, which implies
$\Phi_2 <\Phi_4<...<\Phi_{2p} <  \Phi_{2p+2}=\Phi_1=0. $ That is, the eigenvalues of the  block  $\left[\frac{\partial^2 U}{\partial \vp_i\partial\vp_j} \right]$ consist of one zero, and $n-1$ negative numbers.

 Part 2. The eigenvalues of $\left[\frac{\partial^2 U}{\partial \th_i\partial\th_j} \right]$ at $\bar \q$.   Note that $\th_i=\frac{\pi}{2}, i=1, ..., n$ and $d_{ij}=\min\{ |\vp_i-\vp_j|, 2\pi-|\vp_i-\vp_j|\}$. 
 By equation \eqref{equ:matrix}, the block  has elements 
 \[   \frac{\partial^2 U}{\partial \th_k\partial\th_j} =  \begin{cases}   \frac{1 }{|\sin (j-k)\phi |^3}    \  \  & {\rm for }
 \ j\ne k,\\
- \sum_{i\ne k}  \frac{\cos (i-k)\phi }{|\sin (i-k)\phi |^3}
 \    \ \   \    &{\rm for }
 \ j= k,  \end{cases} \]	
 Note that the block is circulant and symmetric. So its eigenvalues, denoted by $\Theta_1, ..., \Theta_n$,   can be computed by formula \eqref{equ:circu} and are real.  Then, 
 \begin{equation} \notag
 \begin{split}
 \Theta_k&=  \sum_{j= 1}^{2p}\frac{- \cos j \phi}{|\sin j \phi|^3} + \sum_{j= 1}^{2p}\frac{\cos (k-1)j \phi}{|\sin j \phi|^3} =2\sum_{j=1}^p \frac{\cos j(k-1)\phi-\cos j\phi}{\sin^3 j\phi}, \ k=1, ..., n.
 \end{split}
 \end{equation}
 Obviously, $\Theta_2= \Theta_{2p+1}=0$, which reflects the  symmetry,  and
  \[  \Theta_1  =2\sum_{j=1}^p \frac{1-\cos j\phi}{\sin^3 j\phi} >0. \]
Note that $\rho_{k-1}$ is   the complex conjugation of $\rho_{n-(k-1)} =\rho_{n+2-k-1}$ and that  $\left[\frac{\partial^2 U}{\partial \th_i\partial\th_j} \right]$  is  a real matrix. We have 
 $ \Theta_{n+2-k} = \Theta_k. $
 Thus,  it is enough to study just the following eigenvalues
 \[  \Theta_2, \  \Theta_4, \  \Theta_6, \ ..., \  \Theta_{2p}.     \]

 Firstly, we claim that the sequence
 $ \{\Theta_4-\Theta_2,   \ \Theta_6-\Theta_4, \ ...,\ \Theta_{2p+2}-\Theta_{2p} \}$
 is  concave. By  formula \eqref{equ:trig} and other  elementary trigonometric identities, we have
 \begin{align*}
 &\Theta_{k+2}-\Theta_k=-4\sum_{j=1}^{p}\frac{\sin k  j \phi }{\sin^2 j\phi}, \ \Theta_{k+4}-2\Theta_{k+2}+ \Theta_k
 =-8\sum_{j=1}^{p}\frac{\cos (k+1)  j \phi }{\sin j\phi}, \\
 &\Theta_{k+6}-3\Theta_{k+4}+ 3\Theta_{k+2} -\Theta_k= 16\sum_{j=1}^{p} \sin (k+2)  j \phi\\
  & \ \ \ = 16\sum_{j=1}^{p}  \frac{\sin \frac{(k+2)(p+1)}{2}\phi \sin \frac{(k+2)p}{2}\phi   }{\sin \frac{k+2}{2}\phi }\\
   & \ \ \ = -8 \frac{\cos \frac{(k+2)n}{2}\phi-\cos \frac{k+2}{2}\phi  }{\sin \frac{k+2}{2}\phi } \\
  & \ \ \ =-8 \frac{ (-1)^k-\cos \frac{k+2}{2}\phi }{\sin \frac{(k+2)\pi}{n} }. \\
 \end{align*}
 Thus,  	$\Theta_{k+6}-3\Theta_{k+4}+ 3\Theta_{k+2} -\Theta_k<0$  if $k$ is even and    $k\in [2, 2p-2]$, i.e., the sequence $ \{\Theta_4-\Theta_2,   \ \Theta_6-\Theta_4, \ ...,\ \Theta_{2p+2}-\Theta_{2p} \}$  is concave.

 Secondly, note that the two ends of the sequence are positive.
 \begin{align*}
 & \Theta_{2p+2}-\Theta_{2p}=\Theta_1-\Theta_3  =  2\sum_{j=1}^p \frac{1-\cos 2 j\phi}{\sin^3 j\phi}>0,  \\
 &\Theta_4-\Theta_2=2\sum_{j=1}^p \frac{\cos  3j\phi-\cos  j\phi}{\sin^3 j\phi} = -8\sum_{j=1}^{p}\frac{\cos j\phi- \cos ^3j\phi }{\sin^3 j\phi } = -8\sum_{j=1}^{p}\frac{\cos j \phi}{\sin j\phi },
 \end{align*}
 which is positive according to  \eqref{equ:1}.

 Hence,  the sequence $\{\Theta_4-\Theta_2,   \ \Theta_6-\Theta_4, \ ...,\ \Theta_{2p+2}-\Theta_{2p}\} $  is positive, which implies
$  0= \Theta_2 <\Theta_4<...<\Theta_{2p}<\Theta_{2p+2}=\Theta_1. $ That is,  the eigenvalues of the block  $\left[\frac{\partial^2 U}{\partial \th_i\partial\th_j} \right]$ consist of two zeros, and $n-2$ positive numbers. The maximal eigenvalue is $ \Theta_1=2\sum_{j=1}^p \frac{1-\cos j\frac{2\pi}{n}}{\sin^3 j\frac{2\pi}{n}}$.
 \end{proof}

\subsection{Stability on $\S^1$}

 \begin{proof}[proof of Theorem \ref{thm:stas1}]
 	The equilibrium $X_\a$ is a local minimum of the kinetic energy $\sum_{i=2}^{n}\frac{v_i^2}{2\mathtt M_i}$.  In the coordinates $(u_2, ..., u_n,  g_n)$, the Hessian matrix of $U$ at $\bar \q$  is block diagonal since $\frac{\partial U}{\partial g_n}=0$, i.e, $D^2U =\diag\{ \left[\frac{\partial^2 U}{\partial u_i \partial u_j}\right], 0 \}$.  
 	 The signature of the Hessian matrix is $(n_0, n_+, n_-)=(1, 0, n-1)$ by Proposition \ref{prop:main}. So, the first block $\left[ \frac{\partial^2 U}{\partial u_i \partial u_j}\right]$ has $n-1$ negative eigenvalues. 
 	 Hence, the Hamiltonian $H_1= \sum_{i=2}^{n}\frac{v_i^2}{2\mathtt M_i} -U(u_2, ..., u_n)$ is positive definite at the  equilibrium $X_\alpha$, which implies that $X_\alpha$ is Lyapunov stable in the Hamiltonian system $(H_1, J_1^{-1}(n\a)/SO(2), \omega_1)$. 
\end{proof}

 \subsection{Stability on $\S^2$}


Recall that the reduced Hamiltonian is \[   H_2(u_2, ..., u_n, \vec{\th}, v_2, ..., v_n, \vec{p_\th})=F(\vec{\th}, v_2, ..., v_n)  +  \sum_{i=1}^n \frac{p^2_{\th_i}}{2m_i}   -U(u_2, ..., u_n, \vec{\th}), \  \]
where we denote by $F(\vec{\th}, v_2, ..., v_n)$ the function $\frac{1}{2} \vec{v}^T P\vec{v}$. Explicitly, it is
\[  F(\vec{\th}, v_2, ..., v_n)= \sum_{1\le i\le j\le n-1}v_{i+1}v_{j+1} P_{ij} +\sum_{i=1}^{n-1}v_{i+1} n\a P_{in} +\frac{n^2\a^2}{2} P_{nn},  \]
since $G_n=J_2(Y_\a)=n\a$, and the equilibrium is 
\[ Y_\a = ( \frac{2}{n}\pi,  ...,  \frac{k+1}{n}\pi, ..., \pi, \frac{\pi}{2},...,  \frac{\pi}{2},  0, ..., 0, 0, ..., 0 ), \a\in \R. \]
Linearizing  the flow at $Y_\a$ leads to 
 \[   L= \begin{bmatrix} \frac{\partial^2 H_2}{\partial v_i \partial u_j}&\frac{\partial^2 H_2}{\partial v_i \partial \th_j} &\frac{\partial^2 H_2}{\partial v_i \partial v_j}&\frac{\partial^2 H_2}{\partial v_i \partial p_{\th_j}}\\
\frac{\partial^2 H_2}{\partial  p_{\th_i}\partial u_j}&\frac{\partial^2 H_2}{\partial p_{\th_i} \partial \th_j} &\frac{\partial^2 H_2}{ \partial p_{\th_i}\partial v_j}&\frac{\partial^2 H_2}{\partial p_{\th_i} \partial p_{\th_j}}\\
 -\frac{\partial^2 H_2}{\partial u_i \partial u_j}&-\frac{\partial^2 H_2}{\partial u_i \partial \th_j}&-\frac{\partial^2 H_2}{\partial u_i \partial v_j}&-\frac{\partial^2 H_2}{\partial u_i \partial p_{\th_j}}\\
 -\frac{\partial^2 H_2}{\partial \th_i \partial u_j}&-\frac{\partial^2 H_2}{\partial \th_i \partial \th_j}&-\frac{\partial^2 H_2}{\partial \th_i \partial v_j}&-\frac{\partial^2 H_2}{\partial \th_i \partial p_{\th_j}}
 \end{bmatrix}  = 
 \begin{bmatrix}\O& \frac{\partial^2 F}{\partial v_i \partial \th_j}&\frac{\partial^2 F}{\partial v_i \partial v_j}&\O\\
 \O&\O &\O&M&\\
  \frac{\partial^2 U}{\partial u_i \partial u_j}&\frac{\partial^2 U}{\partial u_i \partial \th_j}&\O&\O\\
  \frac{\partial^2 U}{\partial \th_i \partial u_j}& \frac{\partial^2 U}{\partial \th_i \partial \th_j} -\frac{\partial^2 F}{\partial \th_i \partial\th_j}&- \frac{\partial^2 F}{\partial \th_i \partial v_j}&\O
 \end{bmatrix}, \]
  where $\O$ is the zero block.  Note  that  $\left[\frac{\partial^2 U}{\partial u_i \partial \th_j} \right]|_{Y_\a}=\O$.  Recall that $m_1=...=m_n=1$,  $\th_1=...=\th_n=\frac{\pi}{2}, v_2=...=v_n=0$ at $Y_\a$, and that the matrix $P$ depends on $\{\sin \th_1, ..., \sin \th_n\}$ only (see equations \eqref{equ:kinetic}).  At $Y_\a$,  we get $\frac{\partial^2 F}{\partial v_i \partial \th_j} =\cos \th_j (*)=0$, so $ \left[\frac{\partial^2 F}{\partial v_i \partial \th_j} \right]=\O$; we get  $ \frac{\partial^2 F}{\partial v_{i+1} \partial v_{j+1}}=P_{ij}, i,j <n$, and it is easy to check that $P_{ij}=0$ if $i\ne j$ and $P_{ii}=\frac{1}{\mathtt{M}_{i+1}}$, so $\left[\frac{\partial^2 F}{\partial v_i \partial v_j} \right]=\diag\{ \frac{1}{\mathtt M_2}, \frac{1}{\mathtt M_3}, ...,  \frac{1}{\mathtt M_n}\}$; we get $ \left[ \frac{\partial^2 F}{\partial \th_i \partial\th_j}\right]=\frac{n^2\a^2}{2}\left[\frac{\partial^2 P_{nn}}{\partial \th_i \partial\th_j}\right]=\a^2I_n$, since $P_{nn}=\frac{\sum \frac{1}{\sin^2 \th_i}}{n^2}$.  Thus, 
 \[ L|_{Y_\a}=  \begin{bmatrix}\O& \O &\tilde M_1&\O\\
 \O&\O &\O&I_n&\\
 \frac{\partial^2 U}{\partial u_i \partial u_j}&\O&\O&\O\\
 \O& \frac{\partial^2 U}{\partial \th_i \partial \th_j} -\a^2 I_n&\O&\O
 \end{bmatrix}, \ {\rm where} \ \tilde M_1=\diag\{ \frac{1}{\mathtt M_2}, \frac{1}{\mathtt M_3}, ...,  \frac{1}{\mathtt M_n}\}. \]

 \begin{proposition}\label{prop:evalue}
For a block matrix in the form of $  \begin{bmatrix}\O& \O & D&\O\\
\O&\O &\O&E&\\
K&\O&\O&\O\\
\O& Q&\O&\O
\end{bmatrix}, $  suppose that  $D$ and $E$ are invertible, $D, K$ (resp. $Q, E$)are of the same size. 
	If $\u$ is an eigenvector of $KD$  (resp. $QE$) with eigenvalue $\lambda \neq 0$, then there is a two-dimensional invariant subspace    on which  the matrix  is
similar to $\begin{bmatrix}
\sqrt{\lambda}&0\\0&-\sqrt{\lambda}
\end{bmatrix}$.
If $\u$ is an eigenvector of $KD$  (resp. $QE$) with eigenvalue  $0$, then there is a  two-dimensional  invariant subspace   on which  the matrix is similar to 
$\begin{bmatrix}
0&1\\0&0
\end{bmatrix}$.
 \end{proposition}

  Actually,  if $ KD\u=\lambda \u$ and  $\lambda\ne 0$,  then the  basis of the  two-dimensional invariant subspace is $\{(\frac{D\u}{\sqrt{\lambda}}, \0,  \u,  \0),  (-\frac{D\u}{\sqrt{\lambda}}, \0,  \u,  \0)\}$. If $ KD\u=\0$, then the  basis of the  two-dimensional invariant subspace is $\{(D\u,\0,  \0,  \0),  (\0, \0,  \u,  \0)\}$.

   \begin{proof}[proof of Theorem \ref{thm:stas2}]


  By the proof of Theorem \ref{thm:stas1}, the matrix  $\left[\frac{\partial^2 U}{\partial u_i \partial u_j}\right]$ has $n-1$ negative eigenvalues. Note that $\tilde{M}_1$ is positive definite and diagonal, so
  \[  \left[ \frac{\partial^2 U}{\partial u_i \partial u_j} \right]\tilde{M}_1 = (\tilde{M}_1)^{-\frac{1}{2}}\tilde{M}_1 ^{\frac{1}{2}}\left[\frac{\partial^2 U}{\partial u_i \partial u_j}\right](\tilde{M}_1^{\frac{1}{2}})^T\tilde{M}_1^{\frac{1}{2}}. \]
  The above equality 
  and Proposition \ref{prop:evalue} implies that  there is a $(2n-2)$-dimensional invariant subspace of $L|_{Y_\a}$  on which $L|_{Y_\a}$ is semi-simple and  has only non-zero purely imaginary eigenvalues. 
  
 By Proposition  \ref{prop:main},  the  matrix  $ \left[\frac{\partial^2 U}{\partial \th_i \partial \th_j}\right] -\a^2 I_n$ has   eigenvalues:
 $\Theta_1-\a^2, \Theta_2-\a^2, ..., \Theta_n-\a^2. $
 Recall also that $ \Theta_1=2\sum_{j=1}^p \frac{1-\cos j\frac{2\pi}{n}}{\sin^3 j\frac{2\pi}{n}}>0$,  $\Theta_2= \Theta_{n}=0$ and $\Theta_k>0,  k\ne 1, 2, n$. 
 By Proposition \ref{prop:evalue} and the fact that $ \left[\frac{\partial^2 U}{\partial \th_i \partial \th_j}\right] -\a^2 I_n$  is symmetric,  we obtain the Jordan normal form of $L|_{Y_\a}$  on 
 the complementary $2n$-dimensional subspace,  
   \begin{align*}
   &\diag\Big\{\sqrt{\Theta_1},  -\sqrt{\Theta_1}, \begin{bmatrix}
   0&1\\0&0
   \end{bmatrix},   \sqrt{\Theta_3},  -\sqrt{\Theta_3}, ...,  \begin{bmatrix}
   0&1\\0&0
   \end{bmatrix}\Big\},  &{\rm  if \  }  \a^2 =0, \\
 &\diag\Big\{ \begin{bmatrix}
 0&1\\0&0
 \end{bmatrix},\sqrt{\Theta_2-\a^2},  -\sqrt{\Theta_2-\a^2}, 
 ..., \sqrt{\Theta_n-\a^2},  -\sqrt{\Theta_n-\a^2}\Big\},  &{\rm  if \  }  \a^2 =\Theta_1;\\
& \diag\Big\{ \sqrt{\Theta_1-\a^2},  -\sqrt{\Theta_1-\a^2}, 
..., \sqrt{\Theta_n-\a^2},  -\sqrt{\Theta_n-\a^2}\Big\}, & \ {\rm  if \  }  \a^2 \ne \Theta_1, 0. 
  \end{align*}
  This implies that $Y_\alpha$ is linearly unstable in the Hamiltonian system $(H_2, J_2^{-1}(n\a)/SO(2), \omega_2)$   if $\a^2 < \Theta_1=   2\sum_{j=1}^p \frac{1-\cos j\phi}{\sin^3 j\phi}$.

   On the other hand, the form of $L|_{Y_\a}$ implies that 
    \[ D^2 H_2|_{Y_\alpha} =  \diag\Big\{ -\left[ \frac{\partial^2 U}{\partial u_i \partial u_j}\right], \ \a^2 I_n-\left[ \frac{\partial^2 U}{\partial \th_i \partial \th_j}\right], \  \tilde M_1, \ I_n  \Big\}.  \]
    If $\a^2 > \Theta_1=   2\sum_{j=1}^p \frac{1-\cos j\phi}{\sin^3 j\phi}$, then $H_2$ is positive definite at the equilibrium $Y_\alpha$, which implies that $Y_\alpha$ is Lyapunov stable in the reduced Hamiltonian system. 
   
   \end{proof}

Consider the relative equilibria associated with  equilibrium configurations on $\S^1$ discussed in Remark \ref{rem:Nearequilibrium configurations}, i.e., those close to the regular polygonal ones.  Obviously, their stability depends on the two matrices, $ \left[\frac{\partial^2 U}{\partial \vp_i \partial\vp_j}\right]$, $ \left[\frac{\partial^2 U}{\partial \th_i \partial\th_j}\right]$.  By continuity, the eigenvalues of the two blocks  are close to that of the regular polygonal equilibrium configurations. 

\begin{corollary}
	Let $n=2p+1, p\ge 1$.  Let  $\q$ be an  equilibrium configuration on $\S^1$ sufficiently close to $\bar \q$. 
Let $Y_\a\in (H_2, J_2^{-1}(\sum m_i\a)/SO(2), \omega_2)$ be the equilibrium corresponding to the relative equilibrium $A_\a(t)\q$ of the unreduced system.  Then  the equilibrium $Y_\a$ of the reduced Hamiltonian system $(H_2, J_2^{-1}(\sum m_i\a)/SO(2), \omega_2)$ 
is linearly   unstable if  $\a ^2$ is smaller than a certain positive value, and  is Lyapunov stable if  $\a ^2$ is larger then that value. 
\end{corollary}

In the case of three bodies on $\S^1$, this  stability property holds for   all  equilibrium configurations, \cite{DSZ16}.  However, for $n\ge 5$, we are unable to  how to extend Proposition \ref{prop:main} 
to  all the $n$-body  equilibrium configurations on $\S^1$ for now. 

\begin{remark}\label{rem:bif}  
	Another interesting fact is that the relative equilibrium with the critical angular velocity $\a=\pm\sqrt{\Theta_1}$ is the intersection of two families  of relative equilibria for masses $m_1=...=m_n=1, n=2p+1$.  One family is those we have discussed in this paper, namely, those on $\S^1$,  $A_{\a, 0}\bar \q$, $\a \in \R$. 
	For each of the second family, the masses are equally distributed on the circle $x^2+y^2=\sin^2\th$, 
	 $ \th \in (0, \pi)$. Unlike the first family, the angular velocity is determined by $\th$. Actually, the corresponding configuration is a critical point of  $U+\frac{\a^2}{2}\sum_{i=1}^{n} x_i^2+y_i^2=U+\frac{\a^2}{2}\sum_{i=1}^{n} \sin^2 \th_i$, \cite{DSZ17}.   Thus, the angular velocity is (cf. equations \eqref{equ:matrix}), 
	\begin{align*}
\a^2(\th) &=- \frac{\partial U}{\partial \th_1}/ \sin \th_1 \cos \th_1   = \sum_{j=2}^n \frac{1}{\sin^3 d_{ij}} \frac{\sin \th_1 \cos \th_j - \sin \th_j \cos \th_1 \cos (j\frac{2\pi}{n} -\frac{2\pi}{n})}{\sin \th_1 \cos \th_1  }  \\
&= \sum_{j=2}^n \frac{1-\cos (j\frac{2\pi}{n} -\frac{2\pi}{n})}{\sin^3 d_{1j}}
	\end{align*}
As $\th\to \frac{\pi}{2}$, the circle approaches the equator, and the angular velocity approaches $ \pm\sqrt{\Theta_1}$ since $d_{1j} \to (j-1)\phi$.  Thus, the second family intersects the first family at $A_{\pm\sqrt{\Theta_1}, 0}(t)\bar \q$. 
In other words, there is a bifurcation going on. One can read Stoica \cite{Sto18} for more discussion on this bifurcation. 
\end{remark}



 \section{acknowledgments}
 The authors are deeply indebted to Juan Manuel S\'{a}nchez-Cerritos   and Cristina Stoica for suggesting the study of the stability problem of the regular polygonal configurations.   Shuqiang Zhu would like to thank  Florin Diacu for stimulating   interest in mathematics, for  his mentoring and  constant encouragement.   
Xiang Yu is supported by  NSFC(No.11701464) and the Fundamental Research Funds for the Central Universities (No.JBK1805001).  Shuqiang Zhu is supported by NSFC(No.11721101) and funds from China Scholarship Council (CSC NO. 201806345013).

\end{document}